\documentclass[11pt]{amsart}

\usepackage[utf8]{inputenc}
\usepackage[a4paper,twoside]{geometry}

\usepackage[usenames,dvipsnames]{xcolor}

\usepackage{enumitem}
\usepackage{amsmath, amssymb, amsthm}

\usepackage[english]{babel}
\usepackage{fancyhdr}
\usepackage[all]{xy}

\usepackage{tikz}
\usetikzlibrary{calc,decorations.markings,patterns}

\usepackage[framemethod=tikz]{mdframed}

\usepackage{float}
\usepackage{multirow}

\usepackage[colorlinks=true,linkcolor=blue,citecolor=red,backref=page,pdftex]{hyperref}

\newcommand{\unp}{\mathbf{\mathrm{1 \kern-0.25em I}}}
\newcommand{\un}{\mathbf{1}}

\newcommand{\X}{\mathbf X}

\newcommand{\Y}{\mathbf Y}
\newcommand{\R}{\mathbb R}
\newcommand{\N}{\mathbb N}

\newcommand{\di}{{\rm d}}

\newcommand{\cal}{\mathcal}

\newcommand{\lu}{{\mathcal F^1_u}}
\newcommand{\lup}{{\mathcal F^p_u}}

\newcommand{\esp}{\mathbb E}
\newcommand{\prob}{\mathbb P}

\newcommand{\eupp}{\mathcal E_u^p}
\newcommand{\eupa}{\mathcal E^{1+\alpha}_u}

\newcommand{\eupda}{\mathcal E^{2+\alpha}_u}

\newcommand{\eup}{\mathcal E_u^1}

\newcommand{\supp}{\mathrm{supp}\,}

\newtheorem{theorem}{Theorem}[section]
\newtheorem*{theorem*}{Theorem}
\newtheorem*{lemma*}{Lemma}
\newtheorem*{corollary*}{Corollary}
\newtheorem{lemma}[theorem]{Lemma}
\newtheorem{proposition}[theorem]{Proposition}
\newtheorem*{proposition*}{Proposition}
\newtheorem{corollary}[theorem]{Corollary}

\theoremstyle{definition}
\newtheorem{definition}[theorem]{Definition}
\newtheorem{example}[theorem]{Example}

\newtheorem*{hypothesis*}{Hypothesis}
\newtheorem*{notations*}{Notations}
\theoremstyle{remark}
\newtheorem{remark}[theorem]{Remark}

\numberwithin{equation}{section}

\begin{document}

\title[Induction of Markov chains]{Induction of Markov chains, drift functions and application to the LLN, the CLT and the LIL with a random walk on $\R_+$ as an example.}


\author{Jean-Baptiste Boyer}
\email{jean-baptiste.boyer@math.u-bordeaux.fr}


\date{\today}


\begin{abstract} Let $(X_n)$ be a Markov chain on a standard borelian space $\X$. Any stopping time $\tau$ such that $\esp_x\tau$ is finite for all $x\in\X$ induces a Markov chain in $\X$. In this article, we show that there is a bijection between the invariant measures for the original chain and for the induced one.

We then study drift functions and prove a few relations that link the Markov operator for the original chain and for the induced one. The aim is to use this drift function and the induced operator to link the solution to Poisson's equation $(I_d-P)g=f$ for the original chain and for the induced one.

We also see how drift functions can be used to control excursions of the walk and to obtain the law of large numbers, the central limit theorem and the law of the iterated logarithm for martingales.

We use this technique to study the random walk on $\R_+$ defined by $X_{n+1} = \max( X_n + Y_{n+1}, 0)$ where $(Y_n)$ is an iid sequence of law $\rho^{\otimes \N}$ for a probability measure $\rho$ having a finite first moment and a negative drift.
\end{abstract}

\maketitle

\tableofcontents

\section*{Introduction}
For a standard Borelian space $\X$, let $\theta$ be the shift on $\X^{\N}$.

Let $(X_n)$ be a Markov chain on $\X$. We define a Markov operator on $\X$ setting, for a borelian function $f$, $Pf(x)=\esp[f(X_1)|X_0=x]=\esp_x f(X_1)=\int f(X_1)\di\prob_x((X_n))$.

Given a stopping time $\tau$ such that for any $x\in\X$, $\prob_x(\tau<\infty)=1$, we can study the Markov chain $(X_{\tau^n})_{n\in\N}$ where $\tau^n$ is defined by
\[
\left\{\begin{array}{rl}\tau^0((X_n))&=0\\
\tau^{k+1}((X_n)) &=\tau^{k}((X_n))+ \tau(\theta^{\tau^{k}((X_n))}(X_n))
\end{array}\right.
\]
We call $(X_{\tau^n})$ the induced Markov chain and we note $Q$ the Markov operator associated to it, that is, for a borelian function $g$ on $\X$, 
\[Q(g)(x)=\int_{\{\tau<+\infty\}} g(X_\tau)\di\prob_x((X_n))
\]

We define two operators on $\X$ by setting, for a borelian non negative function~$f$ and $x\in\X$,
\begin{align*}
Sf(x)&=\int_{\{\tau=1\}} f(X_1)\di\prob_x((X_n)) \\
Rf(x) &=\int_{\{\tau<+\infty\}} f(X_0)+\dots + f(X_{\tau-1})\di\prob_x((X_n))
\end{align*}

First, these operators allow us to prove the following lemma which is a gneralization of Kac's lemma for Markov chains.

\begin{lemma*}[\ref{lemma:Kac}]
Let $(X_n)$ be a Markov chain on a complete separable metric space $\X$.

Let $\mu$ be a finite $P-$invariant measure on~$\X$ and $\tau$ a $\theta-$compatible stopping time (see definition~\ref{def:theta_compatible}) such that for $\mu$-a.e $x\in \X$, $\lim_{n\to+\infty}\prob_x(\tau\geqslant n) =0$.

For any non negative borelian function $f$ on $\X$, we have
\[\int_\X f\di\mu=\int_\X SRf\di\mu\]
\end{lemma*}

Actually, the operator $R$ is implicitely defined (and studied) by many authors (see for instance~\cite{Kre85}, the end of \S 3.4, the sub paragraph called ``Induced operators" where he even proves our generalization of Kac's lemma in the case of a stopping time which is the first return to some set) but this functional analytic point of view we developp allows us to deal with it easily and, using the operator $S$, we don't need the stopping time to be the first return to some set.

Moreover, it allows us to prove the following
\begin{corollary*}[\ref{corollaire:mesures}]
Let $(X_n)$ be a Markov chain on a complete separable metric space $\X$ and $\tau$ a $\theta-$compatible stopping time (see definition~\ref{def:theta_compatible}) such that for any $x\in\X$, $\esp_x\tau$ is finite.
Define $P$, $Q$, $R$ and $S$ as previously and assume that $QR1$ is bounded on $\X$.

\medskip
Then, $S^\star$ and $R^\star$ are reciproqual linear bijections between the $P-$invariant finite measures and the $Q-$invariant ones which preserve ergodicity.
\end{corollary*}

\medskip
For us, the aim is to prove the central limit theorem for which a standard technique is to find a solution of the \emph{``Poisson equation"} $g-Pg=f$ for $f$ in a certain Banach space.

Indeed, Maigret showed in~\cite{Mai78} that if the chain is positively Harris recurrent with unique invariant measure $\mu$ and $f\in\mathrm L^2(\X,\mu)$ is such that there exist $g\in\mathrm L^2(\X,\mu)$ with $f=g-Pg$, then the CLT holds for $f$ and $\mu-$a.e starting point.

The first idea to solve this equation was to find a spectral gap in some nice space. In particular in $\mathrm{L}^2(\X,\mu)$ where $\mu$ is an invariant measure. This is used by many authors (see for instance~\cite{Ros71}).

An other idea, is to use a stopping time to induce a Markov chain whose study is easier. Glynn and Meyn, in~\cite{GM96}, used this to get the CLT for any starting point assuming the chain is irreducible and there exist a so called \emph{petite set} (a set where the induced chain becomes iid). They used a stopping time which was the time of first return to the petite set to create a pseudo-atom (see~\cite{Num78} and~\cite{Che99}) and find a solution to Poisson's equation.

This is this technique we study in section~\ref{section:drift} and where we show how to find a solution for the Poisson's equation for the original chain when we can solve it for the induced one when there is what we call a ``drift function" (see proposition~\ref{proposition:relations_operateurs} and remark~\ref{remark:poisson}). Actually, when the petite set is compact, one can prove that the induced operator is quasi-compact.

\medskip
We then use this relations to prove a Law of large numbers for martingales (proposition~\ref{proposition:lln_martingales}) and that what Brown calls the Lindeberg condition in~\cite{Bro71} holds in lemma~\ref{lemma:lindeberg}. This is a first step towards the central limit theorem and the law of the iterated logarithm for martingales as we see in corollary~\ref{corollary:TCL_martingales}.

\medskip
Finally,we use this technique to study the random walk on $\R_+$ defined by 
\[
X_{n+1} = \max( X_n + Y_{n+1}, 0)
\]
where $(Y_n)$ is an iid sequence of law $\rho^{\otimes \N}$ for a probability measure $\rho$ having a quadratic moment and a negative drift and we will prove the following

\begin{proposition*}[\ref{proposition:LLN_CLT_LIL_marche_max}]
Let $\rho$ be a probability measure on $\R$ having a finite first moment and a negative drift $\lambda = \int_\R y\di\rho(y)$.

Then, the random walk on $\R_+$ have a finite invariant probability measure $\mu$.

Moreover, for any $\alpha\in \R_+$, if there is $\varepsilon \in \R_+^\star$ such that $\int_\R |y|^{2+\alpha + \varepsilon} \di\rho(y)$ is finite, then for any $f \in \mathcal{F}_{u_{\alpha}}^{1}$ and any $x\in \R_+$,
\[
\frac 1 n \sum_{k=0}^{n-1} f(X_k) \xrightarrow[n\to +\infty]\, \int f \di \mu \; a.e. \text{ and in }\mathrm{L}^1(\prob_x)
\]

Finally, if there is $\varepsilon\in \R_+^\star$ such that $\int_\R |y|^{4+\alpha + \varepsilon} \di\rho(y)$ then for any $f\in \mathcal{E}_{u_\alpha}$ and any $x\in \R_+$, 
\[
\sigma^2_n(f,x)= \frac 1 n \esp_x \left|\sum_{k=0}^{n-1} f(X_k) - \int f \di \mu \right|^2 
\]
converges to some $\sigma^2(f)$ and if $\sigma^2(f) \not =0$, then
\[
\frac 1 {\sqrt n} \sum_{k=0}^{n-1} f(X_k) \xrightarrow {\cal L} \cal N\left( \int f\di\mu, \sigma^2(f) \right)
\]
and
\[
\limsup \frac{\sum_{k=0}^{n-1} f(X_k) -\int f\di \mu}{\sqrt{2n \sigma^2(f) \ln\ln(n)}} =1 \;a.e.\text{ and }\liminf \frac{\sum_{k=0}^{n-1} f(X_k) - \int f \di \mu}{\sqrt{2 n \sigma^2(f) \ln\ln(n)}} =-1 \; a.e.
\]
\end{proposition*}

\section{Induced Markov chains}
\label{section:induction}

\subsection{Definitions}

Let $(X_n)$ be a Markov chain on a standard Borel space $\X$. We define a Markov operator on $\X$ setting, for a borelian function $f$ and $x\in\X$, \[Pf(x)=\esp[f(X_1)|X_0=x]\]

Given a stopping time $\tau$, we can study the Markov chain $(X_{\tau^n})_{n\in\N}$ where $\tau^n$ is defined by
\[\left\{\begin{array}{rl}
\tau^0((X_n))&=0\\
\tau^{k+1}((X_n)) &=\tau^{k}((X_n))+ \tau(\theta^{\tau^{k}((X_n))}(X_n))
\end{array}\right.\]
where $\theta$ stands for the shift on $\X^\N$.

We note $Q$ the sub-Markov operator associated to $(X_{\tau^n})$, that is, for a borelian function $g$ on $\X$ and $x\in\X$, 
\[
Q(g)(x)=\int_{\{\tau<+\infty\}} g(X_\tau)\di\prob_x((X_n))
\]
If, for any $x\in\X$, $\prob_x(\tau\text{ is finite})=1$, then $Q$ is a Markov operator.

Finally, we define two other operator on $\X$ setting, for a borelian non negative function~$f$ and $x\in\X$,
\begin{eqnarray}
Sf(x)=\int_{\{\tau=1\}} f(X_1)\di\prob_x((X_n)) \label{eqn:S} \\
Rf(x) =\int_{\{\tau<+\infty\}} f(X_0)+\dots + f(X_{\tau-1})\di\prob_x((X_n))\label{eqn:R}
\end{eqnarray}

\begin{definition}[$\theta-$compatible stopping times]~\label{def:theta_compatible}\\
We say that a stopping time $\tau$ is $\theta-$compatible if for all $x\in\X$, $\prob_x(\{\tau=0\})=0$ and for $\prob_x-$a.e. $(X_n)\in\X^{\N}$, $\tau((X_n))\geqslant 2$ implies that $\tau(\theta(X_n))=\tau((X_n))-1$. 
\end{definition}

\begin{example}
If $\tau$ is bounded (there exits $M\in\R$ such that for any $x\in \X$ and $\prob_x-$a.e $(X_n)\in \X^\N$, $\tau \leqslant M$), then $\tau$ is not $\theta-$compatible.
In particular, for any stopping time $\tau$ and any $n\in \N$, $\min(\tau,n)$ is not $\theta-$compatible.
\end{example}

\begin{example}\label{example:linearly_recurrent}
Let $\Y$ be a borelian subset of $\X$ and $\tau_{\Y}$ the time of first return in $\Y$:
\[\tau_{\Y}((X_n)) =\inf \{n\in\N^\star ;\; X_n\in \Y\}\]
Then, $\tau_\Y$ is $\theta-$compatible.

Moreover, $\tau^n_{\Y}$ as we defined it coresponds to the time of $n$-th return to $\Y$.

For $x\in \X$, we set $u(x)= \esp_x \tau_{\Y}$ and we call $\Y$ strongly Harris-recurrent if $u$ is finite on~$\X$. This imply in particular that for any $x$ in $\X$, $\tau\Y$ is $\prob_x-$a.e. finite.

Indeed for any borelian non negative function $f$ and any $x\in \X$, we have that

\begin{flalign*}
Qf(x)&= \int_{\{\tau<+\infty\}} f(X_\tau)\di\prob_x = \sum_{n=1}^{+\infty} \esp_x f(X_n)\un_{\{\tau=n\}} &\\
& = \sum_{n=1}^{+\infty} \esp_x f(X_n)\un_{\Y^c}(X_1)\dots \un_{\Y^c}(X_{n-1})\un_{\Y}(X_n) &\\
& = \sum_{n=1}^{+\infty} (P\un_{\Y^c})^{n-1}P(f\un_{\Y})= \sum_{n=0}^{+\infty} (P\un_{\Y^c})^{n}P(f\un_{\Y})& 
\end{flalign*}
\begin{flalign*}
Rf(x)&=\int_{\{\tau<+\infty\}} f(X_0)+\dots+f(X_{\tau-1}) \di\prob_x= \sum_{n=0}^{+\infty}\esp_x f(X_n)\un_{\{\tau\geqslant n+1\}} &\\
&= f(x)+\sum_{n=1}^{+\infty} \esp_x f(X_n)\un_{\Y^c}(X_1)\dots \un_{\Y^c}(X_n)& \\
&= \left(f(x)+\sum_{n=1}^{+\infty} (P\un_{\Y^c})^n(f)(x)\right) =  \sum_{n=0}^{+\infty} (P\un_{\Y^c})^n(f)(x) &
\end{flalign*}
\begin{flalign*}
Sf(x)&=\int_{\{\tau=1\}} f(X_1)\di\prob_x=\int f(X_1)\un_{\Y}(X_1)\di\prob_x=P(f\un_{\Y})& 
\end{flalign*}
Thus, we have that $(R+Q)f=(I_d+RP)f$, $RSf=Qf$, $(P-S)Qf=P(\un_{\Y^c}Qf)=Qf-Sf$ and $(P-S)Rf=P(\un_{\Y^c}Rf)=Rf-f$.

Note that $P,Q,R,S,P-S,Q-S$ and $R-I_d$ are positive operators and so the computations we made make sense for any non negative borelian function $f$.
\end{example}

Next lemma generalizes those relations for any $\theta-$compatible stopping time.
\begin{proposition}
\label{proposition:relation_operateur_positive_functions}Let $\tau$ be a $\theta-$compatible stopping time such that for any $x\in\X$, $\tau$ is $\prob_x-$a.e. finite.

For any non negative borelian function $f$ on $\X$, we have : 
\begin{align*}
(R+Q)f&=(I_d+RP)f\\
(I_d+PR)f &= (I_d+S)Rf \\
(I_d+S)Q f &=(S+PQ)f \\
RS f &= Qf 
\end{align*}
\end{proposition}

\begin{proof}
Let $f$ be a borelian non negative function on $\X$ and $x\in \X$.

Using the Markov property and $\tau$ being a $\theta-$compatible stopping time, we have that for any $n\in \N^\star$,
\[
\esp_x f(X_n)\un_{\{\tau \geqslant n\}} = \esp_x Pf(X_{n-1}) \un_{\{\tau \geqslant n\}}
\]
And so,
\begin{flalign*}
(R+Q)f(x)&=\esp_x f(X_0)+\dots+f(X_\tau) \di\prob_x = \esp_x \sum_{n=0}^{+\infty} f(X_n)\un_{\{\tau\geqslant n\}}& \\
&= f(x)+\sum_{n=1}^{+\infty} \esp_x Pf(X_{n-1})\un_{\{\tau\geqslant n\}} = f(x) + RPf(x)
\end{flalign*}

Moreover, as $\tau$ is $\theta-$compatible,
\[
\int_{\{\tau\geqslant 2\}} Rf(X_1)\di\prob_x((X_n)) = \int_{\{\tau\geqslant 2\}} f(X_1) + \dots + f(X_{\tau-1})\di\prob_x((X_n))
\]
Thus,
\begin{flalign*}
f(x)+&PRf(x)  =f(x)+\int_{\{\tau=1\}} Rf(X_1)\di\prob_x ((X_n))+ \int_{\{\tau\geqslant 2\}} Rf(X_1)\di\prob_x((X_n)) & \\
&= f(x)+SRf(x)+ \int_{\{\tau\geqslant 2\}} f(X_1)+\dots+ f(X_{\tau-1}) \di\prob_x((X_n)) & \\
&= SRf(x)+ \int f(X_0)+\dots + f(X_{\tau-1}) \di\prob_x ((X_n))= SRf(x)+Rf(x)&
\end{flalign*}
Then, by definition of $S$, $\int_{\{\tau=1\}} f(X_1) \di\prob_x ((X_n)) = \int_{\{\tau=1\}} f(X_\tau)  \di\prob_x ((X_n))$, so,
\begin{flalign*}
Sf(x)+PQf(x) & =\int \un_{\{\tau =1\}} \left(f(X_1)+Qf(X_1)\right) + \un_{\{\tau\geqslant 2\}} f(X_\tau)  \di\prob_x((X_n))\\
& = SQf(x)+Qf(x) &
\end{flalign*}
Finally, for any $n\in \N^\star$,
\[
\esp_x Sf(X_{n-1})\un_{\{\tau\geqslant n\}} = \int_{\{\tau=n+1\}} f(X_{n+1})
\]
therefore,
\begin{flalign*}
RSf(x) &= \esp_x \sum_{n=1}^{+\infty} Sf(X_{n-1})\un_{\{\tau\geqslant n\}}  =\sum_{n=1}^{+\infty} \esp_x  Sf(X_{n-1})\un_{\{ \tau\geqslant n\}} &\\
&= \sum_{n=1}^{\infty} \esp_x f(X_{n+1})\un_{\{\tau=n+1\}}= Qf(x)&
\end{flalign*}
\end{proof}

\begin{lemma}\label{lemma:Kac}
Let $\mu$ be a finite $P-$invariant measure on~$\X$ and $\tau$ a $\theta-$compatible stopping time such that for $\mu$-a.e $x\in \X$, $\lim_{n\to+\infty}\prob_x(\tau\geqslant n) =0$.

For any non negative borelian function $f$ on $\X$, we have
\[\int_\X f\di\mu=\int_\X SRf\di\mu\]
\end{lemma}

\begin{proof}
According to proposition~\ref{proposition:relation_operateur_positive_functions},
$f+PRf=Rf+SRf$. So, if $Rf\in\mathrm{L}^1(\X,\mu)$, as $\mu$ is $P-$invariant, we get the lemma. 

If $f\not\in \mathrm{ L}^1(\X,\mu)$, we will get the lemma by approximation.

First, we assume that $f$ is bounded. 
In general, $Rf\not\in \mathrm L^1(\X,\mu)$ so, we approximate it with a sequence of integrable functions.

More precisely, for $n\in\N^\star$, we note $R_n$ the operator defined like $R$ but associated to the stopping time $\min(n,\tau)$ (which is not $\theta-$compatible).

That is to say, for a borelian non negative function $f$ and any $x\in \X$,
\[
R_nf(x)=\esp_x \sum_{k=0}^{\min(\tau,n)-1} f(X_k)
\]

As $\{\min(\tau,n)=1\}=\{\tau=1\}$ for $n\geqslant 2$, the operator $S$ associated to $\min(\tau,n)$ does not depend on $n$ for $n\geqslant 2$.

As $\min(\tau,n)$ is not $\theta-$compatible, we can't use proposition~\ref{proposition:relation_operateur_positive_functions},
but we have for $n\geqslant 2$, that
\begin{flalign*}
PR_nf(x)&=\esp_x \sum_{k=0}^{ \min(\tau\circ\theta,n)-1} f(X_{k+1}) \\
&= SR_nf(x)+ \int_{\{\tau\geqslant 2\}} \sum_{k=0}^{\min( \tau-1,n)-1}f(X_{k+1}){\rm d}\prob_x &\\
&= SR_nf(x) + \int_{\{\tau\geqslant 2\}}\sum_{k=1}^{\min(\tau,n+1)-1} f(X_k){\rm d}\prob_x  = SR_n f +R_{n+1}f-f
\end{flalign*}

And, as $f$ is bounded, for any $x\in\X$, $|R_nf(x)|\leqslant n\|f\|_\infty$ and so $R_nf$ is integrable since $\mu$ is a finite measure and,
\begin{flalign*}
\int SR_nf-f\di\mu &= \int PR_nf -R_{n+1} f\di\mu=\int R_nf-R_{n+1} f\di\mu&\\
&=\int f(X_n)\un_{\{\tau\geqslant n\}} \di\prob_x((X_n))\di\mu(x) \\
&= \int Pf(X_{n-1})\un_{\{\tau\geqslant  n\}}\di\prob_x((X_n))\di\mu(x)
\end{flalign*}
So,
\[\left|\int SR_n f-f\di\mu\right|\leqslant \|f\|_\infty \int_X \prob_x(\{\tau\geqslant n\})\di\mu(x)\xrightarrow\, 0 \text{ (by monotone convergence)}\]
and using the monotone convergence theorem, we get the expected result for borelian bounded functions.

If $f$ is not bounded and not negative, we take $(f_n)$ to be an increasing sequence of bounded positive functions which converges to $f$ and we get the expected result by monotone convergence.
\end{proof}

\begin{example}
If $\tau$ is the return time to some strongly Harris-recurrent set $\Y$, then $Sf(x)=P(f\un_\Y)(x)$.
Moreover for every $P-$invariant measure $\mu$ and every $f\in\mathrm L^1(\X,\mu)$, such that $Rf$ is $\mu-$a.e. finite, $\int_\X SRf{\rm d\mu}=\int_\Y Rf{\rm d}\mu$.

In particular, with $f=1$, we have that, $\int_\Y \esp\tau\di\mu=\mu(\X)$. This is Kac's lemma for dynamical systems.
\end{example}

\subsection{Application to the study of invariant measures}

In this subsection, $X$ is a complete separable metric space endowed with it's Borel tribe and ``measure'' stands for ``borelian measure''. We assume that there exist (at least) a $P-$invariant probability measure on~$\X$.

We also fix a $\theta-$compatible stopping time $\tau$ such that for any $x$ in $\X$, $\esp_x\tau$ is finite.

\begin{lemma}\label{lemma:Setoile_mu}
Let $\mu$ be a finite non-zero $P-$invariant borelian measure on $\X$. Then, $S^\star\mu$ is a finite non-zero $Q-$invariant measure on $\X$.

Moreover, $R^\star S^\star \mu=\mu$ and $S^\star\mu$ is absolutely continuous with respect to $\mu$.
\end{lemma}

\begin{proof}
First, for all non negative  $f\in\mathcal B(\X)$ and all $x\in \X$, $Sf(x)\leqslant Pf(x)$.

So, $\int Sf\di\mu\leqslant \int Pf\di\mu=\int f\di\mu$ since $\mu$ is $P-$invariant and $f$ is bounded. And this proves that $S^\star\mu$ is absolutely continuous with respect to $\mu$. So, as Fubuni's theorem proves that it is $\sigma-$additive, $S^\star\mu$ is a finite measure on $\X$.

\medskip
Moreover, we saw in lemma~\ref{lemma:Kac} that for all non negative borelian function $f$ on $\X$, $\int SR f\di \mu=\int f\di\mu$ and this proves that $R^\star S^\star\mu=\mu$. 

\medskip
Then, we need to prove that $S^\star\mu(\X)>0$.
But, for all $x\in \X$, \[P^kS(1)(x)=\esp_x S1(X_k)\geqslant \prob_x(\{\tau=k+1\})\]

So, $\sum_{k=0}^{n-1} P^kS(1)\geqslant \prob_x(\{\tau\leqslant n+1\})$.
And, as $\mu$ is $P-$invariant, taking the integral on both sides, we get that,
\[nS^\star\mu(X)\geqslant \int_{x\in\X}\prob_x(\{\tau\leqslant n\})\di\mu(x) \]
Finally, we use that for $\mu$-a.e. $x\in \X$, $\lim_n\prob_x(\tau\leqslant n) =1$ and the dominated convergence theorem, tells us that $0<\mu(\X)\leqslant \lim n S^\star\mu(\X)$, so $S^\star\mu(\X)>0$.
\end{proof}

\begin{lemma}\label{lemma:Retoile_nu}
Let $\nu$ be a non-zero $Q-$invariant borelian measure on $\X$. Then, $R^\star\nu$ is a non zero $P-$invariant measure on $\X$.

Moreover, $S^\star R^\star\nu=\nu$ and $\nu$ is absolutely continuous with respect to $R^\star\nu$.

Finally, if $QR(1)$ is bounded on $\X$, then $R^\star\nu$ is a finite measure if and only if $\nu$ is.
\end{lemma}

\begin{remark}
The technical assumption $QR1$ bounded on $\X$ is reasonable.

More specifically, using the same notations as in remark~\ref{example:linearly_recurrent},  we call $\Y$ linearily recurrent if $\sup_{y\in \Y} \esp_y\tau_\Y$ is finite.

In this case, $R1(x)= \esp_x\tau_\Y$ and $QR1(x) = \esp_x R1(X_{\tau_\Y}) \leqslant \sup_{y\in \Y} \esp_y\tau_\Y$ since for any $x \in \X$, $\prob_x(X_{\tau_\Y}\in \Y)=1$ be definition of $\tau_\Y$.
\end{remark}

\begin{proof} 
To prove that $R^\star\mu$ is a measure, one just have to prove that it is $\sigma-$additive. 

Let $(A_n)$ be a sequence of pairwise disjoint borelian subsets of $\X$ and $n\in\N$. As $R$ is a linear operator, we have that $\int R(\un_{\cup _{k=0}^n A_k})\di\nu= \sum_{k=0}^n \int R\un_{A_k}\di\nu$, thus, $R^\star\nu$ is finitely additive. But, according to the monotone convergence theorem, the left side of this equation converges to $\int R(\un_{\cup A_n})\di\nu$ and this finishes the proof that $R^\star\nu$ is $\sigma-$additive.

\medskip
Moreover, for all non negative $f\in\mathcal B(\X)$, $f\leqslant Rf$, so $\nu(f)\leqslant \nu(Rf)$ and $\nu$ is absolutely continuous with respect to $R^\star\nu$ and $R^\star\nu(X)>0$.

\medskip
Then, proposition~\ref{proposition:relation_operateur_positive_functions} shows that for any positive borelian function $f$, $Rf+Qf=f+RPf$. Applying this to $f=\un_A$ for some borelian set $A$, and taking the integral over $\nu$, we get that $\int R\un_A+Q\un_A\di\nu=\int \un_A+RP\un_A\di\nu$. But, $\nu$ is $Q-$invariant so if $\nu(A)$ is finite, we get that $\int R\un_A\di\nu=\int RP\un_A\di\nu$. If $\nu(A)$ is  infinite, the result still holds since in this case, $\int R\un_A\di\nu=\nu(A)=Q^\star\nu(A)=\int RP\un_A\di\nu=+\infty$. Thus, for any borelian set $A$, $R^\star \nu(A)=P^\star R^\star\nu(A)$ that is to say, $R^\star\nu$ is $P-$invariant.

\medskip
As $RS=Q$ and $\nu$ is $Q-$invariant, we directly have that $S^\star R^\star\nu=\nu$.

\medskip
For the last point, assume that $QR(1)$ is a bounded function on $\X$.

If $R^\star\nu$ is finite, then so is $\nu$ since $\nu(\X)\leqslant R^\star\nu(\X)$.

Assume that $\nu$ is finite. Then according to Chacon-Ornstein's ergodic theorem (see chapter 3 theorem 3.4 in~\cite{Kre85}), there exist a $Q-$invariant non negative borelian function $g^\star$ such that $\int g^\star\di\nu=\int R1 \di\nu$ and for $\nu-$almost every $x\in \X$,
\[\frac 1 n \sum_{k=0}^{n-1}Q^kR1(x) \xrightarrow\, g^\star(x) \]
And, since $QR$ is bounded on $\X$ and $R1(x)=\esp_x\tau$ is finite, we get that $g^\star(x)\leqslant \|QR\|_\infty$ for $\nu-$a.e $x\in \X$. So, $g\in\mathrm L^\infty(\X,\nu)\subset \mathrm L^1(\X,\nu)$ since $\nu(\X)<+\infty$ and $\int R1\di\nu\leqslant \|QR\|_\infty \nu(X)<+\infty$.
\end{proof}

We saw in the previous lemmas that $R$ ans $S$ act on invariant measure. As they are linear operators and the set of invariant measures is convex, next proposition shows that they also preserve the ergodic measures (in some sense since they do not preserve probability measures).

\begin{corollary}\label{corollaire:mesures}
Let $(X_n)$ be a Markov chain on a complete separable metric space $\X$ and $\tau$ a $\theta-$compatible stopping time such that for any $x\in\X$, $\esp_x\tau$ is finite.
Define $P$, $Q$, $R$ and $S$ as previously and assume that $QR1$ is bounded on $\X$.

\medskip
Then, $S^\star$ and $R^\star$ are reciproqual linear bijections between the $P-$invariant finite measures and the $Q-$invariant ones which preserve ergodicity.
\end{corollary}

\begin{proof}
We already saw in lemma~\ref{lemma:Setoile_mu} and~\ref{lemma:Retoile_nu}  that $S^\star$ (resp $R^\star$) maps the $P-$invariant (resp. $Q-$invariant) finite non zero measures onto the $Q-$invariant (resp. $P-$invariant) ones and that they are reciproqual to each-other.

\medskip
Thus, it remains to prove that the image by $S^\star$ or $R^\star$ of an ergodic measure still is ergodic. To do so, we use the linearity of $S^\star$ and $R^\star$ and that ergodic probability measures are extreme points of the set of invariant probability measures for a Markov chain in a complete separable metric space.

Let $\mu$ be a $P-$ergodic finite non zero measure. We assume without any loss of generality that $\mu$ is a probability measure. We saw in lemma~\ref{lemma:Setoile_mu} that $S^\star \mu$ is a $Q-$invariant non zero finite measure.

Assume that $S^\star\mu=S^\star\mu(\X)(t\nu_1+(1-t)\nu_2)$ where $\nu_1$ and $\nu_2$ are two $Q-$invariant probability measures and $t\in[0,1]$.

Then, we get that $\mu=R^\star S^\star\mu=S^\star \mu(X) (tR^\star\nu_1+(1-t)R^\star\nu_2)$. But $\mu$ is ergodic, so $\frac 1 {R^\star\nu_1(\X)} R^\star\nu_1=\frac 1 {R^\star\nu_2(\X)} R^\star\nu_2$. And applyting $S^\star$ again, we obtain that $\nu_1=\nu_2$, hence, $S^\star\mu$ is $Q-$ergodic.

The same proof holds to show that if $\nu$ is $Q-$ergodic, then $R^\star\nu$ is $P-$ergodic.
\end{proof}

\section{Drift functions} \label{section:drift}

In this part, $(X_n)$ is a Markov chain on complete separable metric space $\X$ and $\tau$ is a $\theta-$compatible stopping time such that for any $x\in\X$, $\esp_x\tau$ is  finite.

\medskip
We are going to study functions that we call ``Drift functions'' and that allow us to control the Markov chain. These functions are implictely defined by many authors and our main reference for their study is the book of Meyne and Tweedie~\cite{MT93}. We introduce a functional analysis point of view that will allow us to exactly link the original operator and the induced one.

\begin{definition}[Drift functions]\label{def:drift}We call drift function any borelian function $u:\X\to[1,+\infty[$ such that $u-Pu$ is bounded from below and $Qu$, $\frac{\esp\tau}{u}$ and $\frac{P(u-Pu+B)}{u-Pu+B}$ are bounded on $\X$ where $B=\sup_\X Pu-u+1$.
\end{definition}

We will explain the techniqual assumptions of this definition in remark~\ref{remark:assumptions_u}
after some other definitions. 

\begin{remark}
We assumed, in the definition of drift functions, that they are finite on $\X$. It is convenient since it avoids things like ``a point $x$ such that $u(x)$ is finite" or ``a measure $\mu$ such that $\mu(\{ u\text{ is finite}\})>0$" however, sometimes, there is a natural function $u$ which is not finite on $\X$ and we may not want to work with the (standard Borelian) space $\{x;u(x)\text{ is finite}\}$.
\end{remark}

\begin{definition}[$(\varphi,\tau)-$reccurence]\label{definition:phi_recurrence}
Let $\tau$ be a stopping time and $\varphi$ be an increasing function $\varphi:\,\N\to\R_+$ such that $\varphi(1)=1$, $\esp_x \varphi(\tau)$ is finite for all $x\in\X$ and $\liminf \frac 1 n \varphi(n)\not=0$ (the last assumption is equivalent to saying there exist $C\in\R_+^\star$ such that for any $n$ in $\N^\star$, $\varphi(n)\geqslant Cn$ therefore, $\esp\tau\leqslant \frac 1 C\esp\varphi(\tau)$).

\medskip
We say that a borelian subset $\Y$ of $\X$ is $(\varphi,\tau)-$recurrent if for all $x$ in $\X$, $\prob_x(X_\tau\in \Y)=1$ and
\[\sup_{y\in \Y} \esp_y \varphi(\tau) <\infty\]
\end{definition}

\begin{lemma}
Let $\Y$ be a $(\varphi,\tau)-$recurrent set.

Then the function $u_\varphi$ defined on $\X$ by $u_\varphi(x)=\esp_x \varphi(\tau)$ is a drift function.
\end{lemma}

\begin{proof}
Let $b_\varphi=1+\sup_{y\in \Y} \esp_y \varphi(\tau)$.

We note $u_\varphi$, the function $u_\varphi(x)=\esp_x \varphi(\tau)$ and we define $R$ and $S$ as in~\S\ref{section:induction}.

By definition, if $\Y$ is $(\varphi,\tau)-$recurrent then, $Su_\varphi\leqslant Qu_\varphi\leqslant b_\varphi$. 

Moreover, \[
Pu_\varphi(x)=\int_{\{\tau\geqslant 2\}} \varphi(\tau-1){\rm d}\prob_x + Su_\varphi(x)\leqslant u_\varphi(x)+b_\varphi-1\] since $\varphi$ is non decreasing. Hence, $1\leqslant u_\varphi-Pu_\varphi+b_\varphi$ and $u_\varphi$ is a drift function.
\end{proof}

\begin{definition} With the same notations as in the previous definition, if $\varphi(n)=n$ (resp. $\varphi(n)=n^2$, $\varphi(n)=a^{1-n}$ for some $a\in]0,\,1[$), we say that $\Y$ is linearly (resp. quadratically,  exponentially) reccurent and that $u_\varphi$ is a linear (resp. quadratic, exponential) drift function.

If there can be no confusion on the stopping time, we simply say that $\Y$ is $\varphi-$recurrent.
\end{definition}

\begin{definition}
For a borelian function $v:\X\to[1,\infty[$ and $p\in[1,+\infty[$, we note 
\[\mathcal F_v^p(\X)= \left\{f:X\to\R;\; f\textnormal{ is borelian and }\sup_{x\in\X}\frac{|f|^p(x)}{v(x)}\textnormal{ is finite}
\right\}\]
This clearly is a vectorial space and we endow it with a norm, setting, for $f\in \mathcal F_v^p(\X)$,
\[
\|f\|_{\mathcal F_v^p(\X)}=\sup_{x\in\X}\frac{|f(x)|}{v(x)^{1/p}}
\]
Thus, $\mathcal F_v^p(\X)$ is a Banach space which countains borelian bounded functions on $\X$.
\end{definition}
\begin{remark}
The use of this space is that, when $v$ is a drift function, we are be able to deal with functions in $\mathcal F_v^p(\X)$ as if they were bounded.

More specifically, with use the drift condition to have a control on what happens during the walk for functions of $\lu$.
\end{remark}
\begin{remark}\label{remark:Lu_Lv}
Let $v_1,v_2:\X\to[1,\infty[$ be two borelian functions. Then $v_1\in\mathcal F_{v_2}^1(\X)$ if and only if $I_d:\, \mathcal F_{v_1}^1(\X)\to \mathcal F_{v_2}^1(\X)$ is continuous.

Thus, if $u:\X\to[1,+\infty[$ is a borelian function such that $Pu-u$ is bounded from above, we note $B_u=\sup Pu-u+1$ and we clearly have that $1\leqslant u-Pu+B_u$.

Morover, if $B'\geqslant B_u$, we get that $1\leqslant u-Pu+B_u\leqslant u-Pu+B'\leqslant (1+B')(u-Pu+B_u)$. Therefore, $\mathcal F_{u-Pu+B_u}^p$ and $\mathcal F^p_{u-Pu+B'}$ are isomorphic as Banach spaces.
\end{remark}

\begin{definition}
Let $u:\X\to[1,+\infty[$ be a drift function.

We note $B_u=\sup Pu-u+1$ and we set, for $p\in[1,+\infty[$, $\eupp(\X)=\mathcal F^p_{u-Pu+B_u}(\X)$.
\end{definition}

\begin{lemma}
Let $u$ be a drift function. The three following assertions are equivalent :
\begin{itemize}
\item The spaces $\lu$ and $\eup$ are isomorphic
\item There is $a\in[0,1[$ and $b\in\R$ such that $Pu\leqslant au+b$
\item $u\in \eup$
\end{itemize}
\end{lemma}

\begin{proof}
That the first and the third assertions are equivalent follows from remark~\ref{remark:Lu_Lv}.

Moreover, if there is $a\in[0,1[$ and $b\in \R$ such that $Pu\leqslant au+b$, then $(1-a)u\leqslant u-Pu+b$ and so, $u\in \eup$.

In the same way, if $u\in\eup$, then write $u\leqslant \|u\|_{\eup}\left(u-Pu+B_u\right)$ and this means, since $\|u\|_{\eup}\not=0$, that 
\[Pu \leqslant \left(1-\frac1{\|u\|_{\eup}}\right)u + B_u\]
which finishes the proof.

Actually, we also proves that if $u\in\eup$ and $\|u\|_{\eup} \geqslant 1$, then $Pu$ is bounded.
\end{proof}

From now on, we fix a drift function $u$ (we assume that one exists).

\begin{remark}
To simplify the lecture of this section, the reader may think of $u$ as being an exponential drift function and thus, forget about the difference between $\eup$ and $\lu$. He may also fix $p=1$.
\end{remark}

\begin{lemma}
\label{lemma:Rcontinu}
The operator $R$ is continuous from $\eup(\X)$ to $\lu(\X)$.
\end{lemma}

\begin{proof}
Let $f\in \eup(\X)$ and $x\in\X$, $R$ is a positive operator, so
\[|Rf(x)| \leqslant R|f|(x) \leqslant\|f\|_{\eup}R(u-Pu+B_u)\]
But proposition~\ref{proposition:relation_operateur_positive_functions} shows that if $Ru(x)$ is finite, then $R(u-Pu)(x)=u(x)-Qu(x)$ (because $u$ and $Qu$ are assumed to be finite on $\X$). We are going to prove that this relation holds even if $Ru(x)$ is not finite. This will be enough to conclude because $Qu$ is bounded, $\esp\tau\leqslant Cu$ for some $C\in\R$ (by definition of $u$) and because of remark~\ref{remark:Lu_Lv}.

Let $A_n=\sum_{k=0}^{n-1} \left(u(X_k)-Pu(X_k)+B_u\right)$.

We assumed that $u-Pu+B_u\geqslant 1$, so $A_{n+1}-A_n=u(X_n)-Pu(X_n)+B_u\geqslant 0$.\\
Let $M_n=A_n-u(X_0)+u(X_n)-nB_u=\sum_{k=0}^{n-1} u(X_{k+1})-Pu(X_k)$, $M_n$ is a martingale.
Moreover, $\esp_x M_1=0$ and so $\esp_x (A_{min(n,\tau)})=u(x)-\esp_x u(X_{min(\tau,n)})+B_u\esp_x min(\tau,n)$.

Therefore, according to the monotone convergence theorem,
\[R(u-Pu+B_u)(x)=\esp_x A_\tau = u(x)-Qu(x)+B_u\esp_x\tau\]
Hence, $|Rf(x)|\leqslant \|f\|_{\eup}(u(x)-Qu(x)+B_u\esp_x\tau)\leqslant  \|f\|_{\eup}(1+B_u\|\esp\tau\|_{\lu})u(x)$ since $u$ is positive and $\esp\tau\in\lu(\X)$ by assumption.

So, we finally get that $\|Rf\|_{\lu}\leqslant (1+B_u\|\esp\tau\|_{\lu})\|f\|_{\eup}$ and this finishes the proof of the lemma.
\end{proof}

\begin{remark}\label{remark:assumptions_u}
Our assumptions in the definition of drift functions exactly say that if $u$ is one, then $\lu$ and $\eup$ are defined, $P:\eup(\X)\to\eup(\X)$, $Q:\lu(\X)\to\mathcal B(\X)$ are continuous and finally, $R(1)\in\lu$.
\end{remark}

\begin{lemma}\label{lemma:Scontinu}
The operator $S$ is continuous from $\lu$ to $\mathcal B(\X)$.
\end{lemma}

\begin{proof}
First, we remark that $S$ is a positive operator and that for any non negative borelian function $g$, $Sg\leqslant Qg$. Therefore, if  $g\in\lu$,
\[|Sg|\leqslant S(|g|)\leqslant \|g\|_{\lu} Su
 \leqslant \|g\|_{\lu}Qu\leqslant \|g\|_{\lu}\|Qu\|_\infty\]
So, $Sg$ is bounded and $\|Sg\|_{\infty}\leqslant \|g\|_{\lu}\|Qu\|_\infty$.
\end{proof}

\begin{proposition} \label{proposition:relations_operateurs}
For any $f\in \eup(\X)$ and $g\in\lu(\X)$, we have
\begin{flalign*}
R(I_d-P)f &= (I_d-Q)f  \\
 (I_d-P)Rf &= (I_d-SR)f \\
(I_d-P)Q g&= S(I_d-Q)g  \\   
RSg  &= Qg
\end{flalign*}
\end{proposition}

\begin{proof}
This is just a consequence of proposition~\ref{proposition:relation_operateur_positive_functions} and the previous lemmas which shows that all the functions $Rf$, $RPf$, $f$, $Qf$, etc. are finite.
\end{proof}

\begin{remark}\label{remark:Q}
As $Q=RS$ in $\lu$, $Q$ acts on $\lu / \ker(S)$ and it's spectrum in $\lu$ is the same as the one in $\lu/Ker(S)$.

So, if we can solve $Rf=g-Qg$ in $\lu/Ker(S)$ we may not solve it in $\lu$ but the functions $Sg$ and $Qg$ will be defined anyway.

Note that, if $\Y$ is a measurable recurrent subset of $\X$ and $\tau$ is the time of first return to $\Y$, then $Ker(S)=\{g\in\lu;\; g=0\text{ on }\Y\}$.
\end{remark}

\begin{remark}\label{remark:poisson}
Let $f\in\eup$, if there exist a bounded function $g\in\lu/\ker S$ such that $g-Qg=Rf$, then $\hat g=Rf+Qg$ is well defined as we saw in the previous remark.

Moreover, it satisfies $\hat g-P\hat g=f$. This is a direct computation using the relations of proposition~\ref{proposition:relations_operateurs} :
\begin{flalign*}
Rf+Qg-P(Rf+Qg)&=(I_d-P)Rf+(I_d-P)Qg = (I_d-SR)f+S(I_d-Q)g\\
&= f + S(g-Qg-Rf)=f
\end{flalign*}
This, proves that if we can solve Poisson's equation for the induced chain, we can find solutions for the original chain. 

Finally, $Rf+Qg\in\lu$.
\end{remark}

The next two lemmas proves that the spaces $\eupp$ and $\lup$ are really close (in some sense) to the spaces of functions that are integrable against the stationnary measures of the original and the induced random walk.

\begin{lemma}\label{lemma:extmesures}
Let $\mu$ be a $P-$invariant non zero finite measure on $\X$.

Then, $I_d:\eupp(\X)\to \mathcal L^p(\X,\mu)$ is continuous.
\end{lemma}

\begin{proof}
{(we use the same idea as Benoist and Quint in~\cite{BQstat3} Lemma 3.8)}~

We assume without any loss of generality that $\mu$ is a probability measure.

\medskip
Let $f\in \eupp$, $x\in \X$ and $n\in\N^\star$, by definition of $\eupp$, $|f|^p(x)\leqslant \|f\|_{\eupp}^p (u-Pu+b)(x)$, so
\[\frac 1 n \sum_{k=0}^{n-1}P^k(|f|^p)(x)\leqslant \frac{\|f\|_{\eupp}^p} n(u-P^nu+nb)\leqslant\|f\|_{\eupp} ^p(\frac 1 n u(x)+b)\]
Then according to Chacon-Ornstein's ergodic theorem (see chapter 3 theorem 3.4 in the book of Krengel~\cite{Kre85}), there exist a $P-$invariant non negative borelian function $f^\star$ such that $\int |f|^p\di\mu=\int f^\star \di\mu$ and for $\mu-$almost every $x\in \X$,
\[\frac 1 n \sum_{k=0}^{n-1}P^k|f|^p(x) \xrightarrow\, f^\star(x) \]
And, since $u(x)$ is finite for all $x\in \X$, we get that $f^\star(x)\leqslant b\|f\|_{\eupp}^p$ for $\mu-$a.e $x\in \X$. So, $f^\star\in\mathrm L^\infty(\X,\mu)\subset \mathrm L^1(\X,\mu)$ since $\mu(\X)<+\infty$, $f\in\mathcal L^p(\X,\mu)$ and $\|f\|_{\mathcal L^p(\X,\mu)}\leqslant b^{1/p}\|f\|_{\eupp}$.
\end{proof}

\begin{lemma}\label{lemma:extmesures_Qinvar}
Let $\nu$ be a $Q-$invariant non zero finite measure on $\X$.

Then, $I_d: \lup(\X)\to \mathcal L^p(\X,\nu)$ is continuous.
\end{lemma}

\begin{proof}
We use the same idea as in lemma~\ref{lemma:extmesures}.

We assume without any loss of generality that $\nu$ is a probability measure.

\medskip
Let $g\in \lup$, $x\in\X$ and $n\in\N^\star$,
\begin{flalign*}
\frac 1 n \sum_{k=0}^{n-1}Q^k|g|^p(x)&\leqslant \frac{\|g\|_{\lup}^p} n\sum_{k=0}^{n-1}Q^k(u-Qu+b\esp\tau) \leqslant \frac{\|g\|_{\lup}^p} n \left(u-Q^nu + \sum_{k=0}^{n-1} Q^k \esp\tau\right)
\end{flalign*}
But, by definition of $u$, $\esp\tau \in\lu$ and for $k\geqslant 1$, $Q^k(u)\leqslant \|Qu\|_\infty<+\infty$, so
\[
\frac 1 n \sum_{k=0}^{n-1} Q^k|g|^p(x)\leqslant \frac{\|g\|_{\lup}^p} n (2u+(n-1)\|Qu\|_\infty)\]
According to Chacon-Ornstein's ergodic theorem, there exist a $Q-$invariant function $g^\star$ such that $\int |g|^p\di\nu=\int g^\star\di\nu$ and for $\nu-$almost every $x\in\X$,
\[\frac 1 n \sum_{k=0}^{n-1}Q^k|g|^p(x) \xrightarrow\, g^\star(x)
\]
As $u$ is finite on $\X$, we get that $g^\star\leqslant \|Qu\|_{\infty} \|g\|_{\lup}^p$, so $g^\star\in\mathrm L^\infty(\X,\nu)\subset \mathrm L^1(\X,\nu)$ and $|g|^p\in\mathrm L^1(\X,\nu)$.

Thus, $\|g\|_{\mathcal L^p(\X,\nu)}=\left(\int g^\star\di\nu\right)^{1/p} \leqslant \|Qu\|_\infty^{1/p} \|g\|_{\lup}$.
\end{proof}

Let
\begin{align*}
\mathcal E_0^{p\star}&=\{T\in(\eupp)^\star;\;\forall f\in\eupp\, T(f)=T(SRf)\}=Ker(I_d-R^\star S^\star)&\\
\mathcal F_0^{p\star}&=\{T\in(\lup)^\star;\;\forall g\in\lup\, T(g)=T(RSg)\}=Ker(I_d-S^\star R^\star)&
\end{align*}

\begin{corollary}\label{corollaire:dual} 
If $u$ is a drift function,
$\mathcal E_0^{p\star}$ and $\mathcal F_0^{p\star}$ are two Banach spaces and, $S^\star:\,\mathcal E_0^{p\star}\longrightarrow \mathcal F_0^{p\star}$ and $R^\star:\, \mathcal F_0^{p\star}\to \mathcal E_0^{p\star}$ are continuous and reciproqual.
\end{corollary}

\begin{proof}
The proof is direct from lemma~\ref{lemma:Rcontinu}, \ref{lemma:Scontinu} and  \ref{proposition:relations_operateurs} .
\end{proof}

\section{The LLN, the CLT and the LIL for martingales bounded by a Drift function}\label{section:LLN_TCL}

First, we extend Breiman's law of large numbers for martingales (see \cite{Br60}) for measurable functions such that $f^{1+\alpha}\in\eup$ for some $\alpha>0$ 
\begin{lemma}\label{lemma:sommation_Abel}
Let $u$ be a drift function, $x\in X$, and $\alpha\in\R_+$,
\[ \sup_{n\in\N}\sum_{k=0}^{n} \frac {P^{k}(u-Pu)}{(k+1)^{\alpha}} \leqslant u(x)\]
\end{lemma}

\begin{proof}
This is a direct computation :
\begin{flalign*}
\sum_{k=0}^{n} \frac {P^{k}(u-Pu)}{(k+1)^{\alpha}} &= \sum_{k=0}^n \frac 1 {(k+1)^\alpha} P^ku -\sum_{k=0}^n \frac 1 {(k+1)^\alpha} P^{k+1} u &\\
& =\sum_{k=1}^n (\frac 1 {(k+1)^\alpha}-\frac 1 {k^\alpha})P^ku + u-\frac 1 {(n+1)^\alpha} P^{n+1} u &\\
& \leqslant u(x) \text{ since }u\text{ is non negative}&
\end{flalign*}
\end{proof}

\begin{lemma}\label{lemma:u(X_n)/n}
Let $u$ be a drift function and $\alpha\in \R_+^\star$. Then, for all $f\in \eup$ and all $x\in \X$ such that $u(x)$ is finite,
\[
\frac 1 {n^{1+\alpha}} f(X_n) \xrightarrow\, 0 \; a.e.\text{ and in }\mathrm{L}^1(\prob_x)
\]
\end{lemma}

\begin{proof}
We compute
\[
\esp_x |f(X_n)| \leqslant \|f\|_u \esp_x u(X_n) - Pu(X_n) + b = P^n(u-Pu+b)
\]
So,
\[
\sum_{k=0}^n \frac{\esp_x |f(X_k)|}{(k+1)^{1+\alpha}} \leqslant \sum_{k=0}^n \frac{P^k(u-Pu)}{(k+1)^{1+\alpha}} + b\sum_{k=0}^n \frac 1 {(k+1)^{1+\alpha}} \leqslant u(x) + b \sum_{n\in \N^\star} \frac 1 {n^{1+\alpha}} 
\]
where we used lemma~\ref{lemma:sommation_Abel} to bound the first sum.

So, for any $x\in \X$ such that $u(x)$ is finite,
\[
\sum_{k=0}^{+\infty} \frac{\esp_x |f(X_k)|}{(k+1)^{1+\alpha}}
\]
is finite.

This proves the convergence in $\mathrm{L}^1(\prob_x)$ and to get the $a.e.-$one, we use Borel-Cantelli's theorem since for any $\varepsilon \in \R_+^\star$,
\[
\sum_{n=0}^{+\infty} \prob_x\left( \frac{|f(X_n)|}{(n+1)^{1+\alpha}} \geqslant \varepsilon \right) \leqslant \frac 1 \varepsilon \sum_{k=0}^{+\infty} \frac{\esp_x |f(X_k)|}{(k+1)^{1+\alpha}}
\]
\end{proof}

\begin{proposition}[Law of large numbers for martingales]~\\
\label{proposition:lln_martingales}
Let $u$ be a drift function and $\alpha\in \R_+^\star$. For all $f\in\eupa$  and all $x\in X$,
\[\frac 1 n \sum_{k=0}^{n-1} f(X_{k+1})-Pf(X_k)\xrightarrow \, 0  \;\;\prob_x-a.e.\textnormal{ and in } \mathrm L^{1+\alpha}(\prob_x)\]
\end{proposition}

\begin{proof}
Let $M_n= \sum_{k=0}^{n-1} f(X_{k+1})-Pf(X_k)$.\\
$M_n$ is a martingale of null expectation.\\
And
\begin{flalign*}
\esp_x |M_{n+1}-M_n|^{1+\alpha} &=\esp_x |f(X_{n+1})-Pf(X_{n})|^{1+\alpha}=P^n(\esp_x|f(X_1)-Pf(x)|^{1+\alpha}) &\\
& \leqslant P^{n+1}(|f|^{1+\alpha})(x)\leqslant\| |f|^{1+\alpha}\|_{u-Pu+b} P^{n+1}(u-Pu+b)& 
\end{flalign*}
Hence,
\begin{flalign*}
\sum_{n=1}^{+\infty} \frac 1 {n^{1+\alpha}} \esp_x |M_{n+1}-M_n|^{1+\alpha} &\leqslant \|f^{1+\alpha}\|_{u-Pu+b}\sum_{n=1}^{+\infty} \frac {P^{n+1}(u-Pu+b)}{n^{1+\alpha}} \\
&\leqslant \|f^{1+\alpha}\|_{u-Pu+b} \left( u(x)+b\sum_{k=0}^{+\infty} \frac 1 {n^{1+\alpha}}\right)
\end{flalign*}
And using the strong law of large numbers for martingales (see theorem 2.18 in \cite{HH80}), we get that $\frac 1 n M_n\xrightarrow \,0$ $\prob_x-$a.e. and in $\mathrm L^{1+\alpha}(\prob_x)$.
\end{proof}

Using the same kind of trick, we can prove the following

\begin{lemma}\label{lemma:lindeberg}
Let $u$ be a drift function and let $\alpha\in\R_+^\star$.

Let $g\in\eupda$ and $x\in\X$ be such that $u(x)$ is finite.

Then, for any $\varepsilon\in \R_+^\star$
\[\frac 1 n \sum_{k=0}^{n-1} \esp_x \left((g(X_{k+1})-Pg(X_k))^2 \un_{ |g(X_{k+1})- Pg(X_k)|\geqslant\varepsilon\sqrt n}\right)\xrightarrow[n\to +\infty]\, 0\]
and
\[
\sum_{n=1}^{+\infty} \frac1 {\sqrt n} \esp_x \left(\left|g(X_{n+1})-Pg(X_n) \right| \un_{ |g(X_{n+1})- Pg(X_n)|\geqslant\varepsilon\sqrt n} \right)\text{ is finite}
\]
Finally, there is $\delta \in\R_+^\star$ such that
\[
\sum_{n=1}^{+\infty} \frac 1 {n^2} \esp_x \left((g(X_{n+1})-Pg(X_n))^4 \un_{ |g(X_{n+1})- Pg(X_n)|\leqslant\delta \sqrt n}\right) 
\]
is finite.
\end{lemma}

\begin{proof}
Using Markov's property and inequality, we have that,
\[
\esp_x \left(h(X_{k+1}, X_k)^2\un_{ |h(X_{k+1}, X_k)|\geqslant\varepsilon\sqrt n}\right)
\leqslant\frac { P^k\left(\esp\left(g(X_1)-Pg(X_0)\right)^{2+\alpha}\right)} {\varepsilon^{\alpha} n^{\alpha/2}}
\]
where we noted $h(x,y)= g(x) - Pg(y)$.

But, $\esp_x\left[\left((g(X_1)-Pg(X_0)\right)^{2+\alpha}\right]\in\eup$, since $g\in\eupda$.

and so,
\begin{flalign*}
\frac 1 n \sum_{k=1}^n \esp_x \left(h(X_{k+1}, X_k)^2\un_{ |h(X_{k+1},X_k)|\geqslant\varepsilon\sqrt n}\right)
&\leqslant  C {n^{1+\alpha/2} \varepsilon^\alpha} \sum_{k=0}^{n-1} {P^k(u-Pu+b) }\\
&\leqslant  \frac C {n^{1+\alpha/2} \varepsilon^\alpha} u(x) + \frac {bC} {n^{\alpha/2} \varepsilon^\alpha}
\end{flalign*}
And the right side of this inequality goes to $0$ when $n$ goes to infinity.

\medskip
The two sums we have to bound are dominated by constants times
\[
\sum_{n=1}^{+\infty} \frac {1} {n^{1+\alpha/2}} \esp_x \left(\left|g(X_{n+1})-Pg(X_n)\right|^{2+\alpha} \right) 
\]
and, once again, using that $g\in \eupda$, we have that
\[
 \esp_x \left(\left|g(X_{n+1})-Pg(X_n)\right|^{2+\alpha} \right)  \leqslant \|g\|_{\eupda}  P^n(u-Pu+b)
\]
And we conclude with lemma~\ref{lemma:sommation_Abel}.
\end{proof}

Lemma~\ref{lemma:lindeberg} is actually the first step in the proof of the Central limit theorem and of the law of iterated logarithm for martingales as we see in next
\begin{corollary}[Central Limit Theorem and Law of the Iterated Logarithm for martingales] \label{corollary:TCL_martingales} Let $u$ be a drift function and $\alpha\in\R_+^\star$.

Let $g\in\eupda$ and $x\in\X$ be such that $u(x)$ is finite.

If
\[
\frac 1 n \sum_{k=0}^{n-1} P(g^2)(X_k) - (Pg(X_k))^2
\]
converges in $\mathrm{L}^1(\prob_x)$ and almost-everywhere, to some $\sigma^2(g,x)\not=0$,
then
\[
\frac 1 {\sqrt n} \sum_{k=0}^{n-1} g(X_{k+1})-Pg(X_k) \xrightarrow[n\to\infty]{\mathcal L}\cal N(0,\sigma^2(g,x))
\]
Moreover,
\[
\limsup \frac{\sum_{k=0}^{n-1} g(X_{k+1}) - Pg(X_k)}{\sqrt{2n \sigma^2(g,x) \ln\ln(n)}} =1 \;a.e.
\]
and
\[
\liminf \frac{\sum_{k=0}^{n-1} g(X_{k+1}) - Pg(X_k)}{\sqrt{2 n \sigma^2(g,x) \ln\ln(n)}} =-1 \; a.e.
\]
\end{corollary}

\begin{proof}
The first convergence is a straightforward consequence of Brown's central limit theorem for martingales (see~\cite{Bro71}) since the so called Lindeberg condition holds when the space has a drift function, according to lemma~\ref{lemma:lindeberg}.

The Law of the iterated logarithm is given by corollary~4.2 and theorem~4.8 of Hall and Heyde in~\cite{HH80} since the assumptions of corollary hold according to lemma~\ref{lemma:lindeberg}.
\end{proof}

\begin{remark}
If $g\in \eupda$, lemma~\ref{lemma:u(X_n)/n} proves that $g(X_n)/\sqrt n$ converges to $0$ in $\mathrm{L}^1(\prob_x)$ and $a.e.$ so the previous results still holds if we look at $\sum_{k=0}^{n-1} g(X_k) - Pg(X_k)$ instead of $\sum_{k=0}^{n-1} g(X_{k+1}) - Pg(X_k)$. We will use this remark in the sequel to study functions $f$ that can be written $f=g-Pg$ with $g\in \eupda$.
\end{remark}

\section{Application to a random walk on a half line}

In this section, we apply the previous properties to a random walk on $\R_+$.

Let $\rho$ be a probability measure on $\R$. For $x_0 \in \R_+$, we define the random walk $(X_n)$ driven by $\rho$ and starting at $x_0$ by
\begin{equation} \label{equation:max_walk}
\left\{\begin{array}{ccc}
X_0 & = & x_0 \\
X_{n+1} & = & \max(X_n + Y_{n+1}, 0)
\end{array}
\right.
\end{equation}
where $(Y_n)$ has law $\rho^{\otimes \N}$.

This random walk on a half-line is a model for storage systems and queueing processes.
In their book~\cite{MT93}, Meyne and Tweedie study the recurrence properties of the walk and in~\cite{La95}, Lalley studies the return probabilities for this walk.

\medskip
It is clear that if $\supp \rho \subset \R_-^\star$, then, for any starting point $x$, and $\rho^{\otimes \N}-$a.e. $(Y_n)$, the walk stays at $0$ after a finite number of steps. But, if $\rho(\R_+^\star) >0$, then the walk is not bounded : for any $M \in \R_+$ and any $x\in \R_+$, $\prob_x( \sup_{n} X_n \leqslant M) =0$.

\medskip
First, we have the following
\begin{proposition}
Let $\rho$ be a probability measure on $\R$ having a finite first moment. 
Then, if $\tau$ is the ($\theta-$compatible) stopping time defined by
\[
\tau((X_n)) = \inf\{ n \in \N^\star | X_{n-1} + Y_{n} \leqslant 0\}
\]
we have that for any $x\in \R_+$, $\esp_x \tau$ (in particular, $\tau$ is a.e. finite).

Moreover, there is a $P-$invariant probability measure $\mu$ on $\R_+$ and it is given, for any borelian bounded function $f$, by
\[
\mu(f) = \frac 1 {\esp_0 \tau} \esp_0 \sum_{k=0}^{\tau -1 } f(X_k) 
\]
\end{proposition}

\begin{proof}
That for any $x\in \R_+$, $\esp_x\tau $ is finite comes from proposition~18.1 in \cite{Spi64}.

Let $Q$ be the induced operator by $\tau$ and $R$ and $S$ be the operators associated to $\tau$ and defined by equations~\ref{eqn:R} and~\ref{eqn:S}. The measure $\mu$ now writes :
\[
\mu(f) = \frac{Rf(0)}{R1(0)}
\]

So, it is a probability measure and according to~\ref{proposition:relations_operateurs} we have that for any borelian bounded function $f$,
\[
R(I_d-P)f(0) = (I_d-Q)f(0) = f(0) - \esp_x(f(X_\tau)) = f(0)-f(0) = 0\]
thus, the measure $\mu$ is $P-$invariant. (to apply proposition~\ref{proposition:relations_operateurs}, we choose $u(x) = \esp_x\tau$ since then, $Pu(x) = \int_{\{\tau \geqslant 2\}} \tau-1 + \int_{\{\tau =1\}} \tau^2 = u(x) -1 + \esp_0 \tau \prob_x(\tau =1)$ and so, $u$ is a drift function).
\end{proof}

\begin{lemma}
Let $s\in [1,+\infty[$ be such that $\int_\R |y|^s \di\rho(y)$ is finite.

Then,
\[
\lim_{x\to +\infty} \int_\R \frac{\max(x+y,0)^s - x^s}{x^{s-1}} \di\rho(y) = s\int_\R y\di\rho(y)
\]
\end{lemma}

\begin{proof}
Let's compute, for $x$ large,
\begin{flalign*}
 \int_\R \frac{\max(x+y,0)^s - x^s}{x^{s-1}} \di\rho(y) & = \int_{-\infty}^{-x} (-x) \di\rho(y) + \int_{-x}^{+\infty} \frac{(x+y)^s-x^s}{x^{s-1}} \di\rho(y) \\
 &= -x \rho(]-\infty,-x]) + \int_{-x}^{+\infty} x \left( \left(1+ \frac y x\right)^s -1 \right) \di\rho(y)
\end{flalign*}
Note, for $u\in [-1,0[ \cup ]0, +\infty[$,
\[
\varphi(u) = \frac{(1+u)^s -1 -su}{u}
\]
Then, $\lim_{u\to 0^+} \varphi(u) = 0$ so we can extend $\varphi$ by continuity at $0$ and there is a constant $C\in \R_+$ such that for any $u\in[-1, +\infty[$, $|\varphi(u)| \leqslant C(1+|u|^s)$. Thus, we have that
\begin{flalign*}
 \int_\R \frac{\max(x+y,0)^s - x^s}{x^{s-1}}& \di\rho(y) =  -x \rho(]-\infty,-x]) + \int_{-x}^{+\infty} x \left( s \frac y x + \frac y x \varphi(\frac y x) \right) \di\rho(y) \\
 &= -x \rho(]-\infty,-x]) + s\int_{-x}^{+\infty} y\di\rho(y) + \int_{-x}^{+\infty} y\varphi(\frac y x) \di\rho(y)
\end{flalign*}
But, for any $x\in[1,+\infty[$,
\[
\un_{]-x, +\infty]}(y) \left|y\varphi(\frac y x)\right| \leqslant C|y|(1+|y|)^{s-1} \in \mathrm{L}^1(\rho)
\]
and for any $y\in \R$, $\lim_{x\to +\infty} \un_{]-x, +\infty]}(y) \left|y\varphi(\frac y x)\right| = 0$, so the dominated convergence theorem proves that
\[
\lim_{x\to +\infty} \int_{-x}^{+\infty} y\varphi(\frac y x) \di\rho(y) = 0
\]
Moreover, as $\int_{\R} |y| \di\rho(y)$ is finite, we have that
\[
\lim_{x\to +\infty} \int_{-\infty}^{-x} |y| \di\rho(y) =0
\]
But,
\[
\int_{-\infty}^{-x} |y| \di\rho(y) \geqslant \int_{-\infty}^{-x} |x| \di\rho(y) = |x| \rho(]-\infty, -x]) \geqslant 0
\]
so $\lim_{x\to +\infty} |x|\rho (]-\infty, -x]) = 0$ and this finishes the proof of the lemma.
\end{proof}

\begin{corollary}
Let $s\in [1, +\infty[$ be such that $\int_\R |y|^s \di\rho(y)$ is finite.

Define, for $x\in \R_+$, $u_s(x) = 1+|x|^s$. 

Then, there are a constants $B_s, N_s$ such that 
\[
\frac{u_{s-1}}{N_s} \leqslant  u_s - Pu_s + B_s \leqslant N_s u_{s-1}
\]
Moreover, if $s\geqslant 2$, then $u_s$ is a drift function and $\mathcal{F}^1_{u_{s-1}} $ is continuously isomorphic to $\mathcal{E}^1_{u_s}$.

In particular, $u_{s-1} \in \mathrm{L}^1(\R_+, \mu)$.
\end{corollary}

\begin{proof}
According to the previous lemma, we have
\[
\lim_{x\to +\infty} \frac{Pu_s(x) - u_s(x)}{u_{s-1}(x)} = s\int_\R y\di\rho(y) <0
\]
so, there are $x_0,\varepsilon \in \R_+^\star$ such that for any $x\geqslant x_0$, $u_s(x) - Pu_s(x) \geqslant \varepsilon u_{s-1}(x)$.

Now, the function $\varepsilon u_{s-1} + Pu_s - u_s$ is continuous on $[0,x_0]$ so it is bounded by some non negative constant that we note $B_s$ and we have that for any $x\in \R_+$,
\[
\varepsilon u_{s-1} \leqslant u_s - Pu_s + B_s
\]
To find the other domination, we apply once again the previous lemma to find that $(u_s - Pu_s + B_s)/u_{s-1}$ is bounded.

If $s\geqslant 2$, as $u_{s-1}$ is non negative, this also proves that $u_s - Pu_s$ is bounded from below. Moreover, $Qu(x) = u(0)$ is finite, $\esp_x \tau = R1(x)$ and $1 \in \mathcal{E}_{u_{s}}$ so $R1 \in \mathcal{F}_{u_{s}}$

Finally,
\begin{flalign*}
P(u_s-Pu_s+B_s) &\leqslant N_s P u_{s-1} \leqslant N_s \left( u_{s-1} + B_{s-1}\right) \\
&\leqslant N_s \left( N_s(u_s-Pu_s + B_s) + B_{s-1} \right)
\end{flalign*}
and this finishes the proof that $u_s$ is a drift function for $s\geqslant 2$.
\end{proof}

\begin{proposition} \label{proposition:poisson_equation_marche_max}
Let $\rho$ be a probability measure on $\R$ having a negative drift and a polynomial moment of order $s$ for some $s\in [1, +\infty[$.

Then, for any $f\in \mathcal{F}^1_{u_{s-1}}$, there is $g\in \mathcal{F}^\infty_{u_s}$ such that $f=g-Pg + \int f\di\mu$.
\end{proposition}

\begin{proof}
Using, the previous corollary, we have that any $f\in \mathcal{ F}^1_{u_{s-1}}$ belong to $\mathcal{E}^1_{u_s}$. Moreover, we have, according to lemma~\ref{lemma:Rcontinu}, that $R$ continuously maps $\mathcal{E}^1_{u_s}$ onto $\mathcal{F}^1_{u_s}$ so, for any $f\in \mathcal{F}^1_{u_{s-1}}$, $Rf \in \mathcal{F}^1_{u_s}$.

Finally, using the equations of proposition~\ref{proposition:relations_operateurs}, we also have that
\[
(I_d-P)Rf = f - SRf
\]
but for any $x\in \R_+$,
\[
SRf(x) = \int_{\{\tau = 1\}} Rf(X_1) \di\prob_x((X_n)) = Rf(0) \prob_x(\tau =1) =  R1(0) \prob_x(\tau =1) \int f\di\mu
\]
since, $\mu(f) = \frac{Rf(0)}{R1(0)}$. So, applying this to $\widehat f =f-\int f \di \mu$, we get that $SR\widehat f=0$, $R\widehat f \in \mathcal{F}^1_{u_s}$ and $(I_d-P)R\widehat f = \widehat f = f - \int f\di\mu$ so $R\widehat f$ is the required solution to Poisson's equation.
\end{proof}

\begin{proposition}\label{proposition:LLN_CLT_LIL_marche_max}
Let $\rho$ be a probability measure on $\R$ having a finite first moment and a negative drift $\lambda = \int_\R y\di\rho(y)$.

Let $\alpha\in \R_+$.

If there is $\varepsilon \in \R_+^\star$ such that $\int_\R |y|^{2+\alpha + \varepsilon} \di\rho(y)$ is finite, then for any $f \in \mathcal{F}_{u_{\alpha}}^{1}$ and any $x\in \R_+$,
\[
\frac 1 n \sum_{k=0}^{n-1} f(X_k) \xrightarrow[n\to +\infty]\, \int f \di \mu \; a.e. \text{ and in }\mathrm{L}^1(\prob_x)
\]

If there is $\varepsilon\in \R_+^\star$ such that $\int_\R |y|^{4+\alpha + \varepsilon} \di\rho(y)$ then for any $f\in \mathcal{E}_{u_\alpha}$ and any $x\in \R_+$, 
\[
\sigma^2_n(f,x)= \frac 1 n \esp_x \left|\sum_{k=0}^{n-1} f(X_k) - \int f \di \mu \right|^2 
\]
converges to some $\sigma^2(f)$ and if $\sigma^2(f) \not =0$, then
\[
\frac 1 {\sqrt n} \sum_{k=0}^{n-1} f(X_k) \xrightarrow {\cal L} \cal N\left( \int f\di\mu, \sigma^2(f) \right)
\]
and
\[
\limsup \frac{\sum_{k=0}^{n-1} f(X_k) -\int f\di \mu}{\sqrt{2n \sigma^2(f) \ln\ln(n)}} =1 \;a.e.\text{ and }\liminf \frac{\sum_{k=0}^{n-1} f(X_k) - \int f \di \mu}{\sqrt{2 n \sigma^2(f) \ln\ln(n)}} =-1 \; a.e.
\]
\end{proposition}

\begin{proof}
The moment assumption and proposition~\ref{proposition:poisson_equation_marche_max} proves that for any $f\in \cal{F}_{u_\alpha}^1$, there is $g\in \cal {F}_{u_{\alpha+1}}$ such that $f-\int f \di\mu = g-Pg$.

Moreover, $\cal {F}_{u_{\alpha+1}} \subset \cal{E}_{u_{\alpha+2}}$ so, for any $\alpha' \in \R_+$, $|g|^{1+\alpha'} \in \cal E_{u_{(1+\alpha')(\alpha+2)}}$. And this means that $|g| \in \cal{E}^{1+\alpha'}_{u_{(1+\alpha')(\alpha+2)}}$.

Therefore, we can write
\[
\sum_{k=0}^{n-1} f(X_k) - \int f\di\mu = g(X_0) - g(X_n) + \sum_{k=0}^{n-1} g(X_{k+1}) - Pg(X_k)
\]
and using lemma~\ref{lemma:u(X_n)/n} and ~\ref{proposition:lln_martingales}, we get the first expected result if there is $\varepsilon \in \R_+^\star$ such that $\int_\R |y|^{2+\alpha+\varepsilon} \di\rho(y)$ is finite and if we choose $\alpha'$ small enough.

Let us now compute, using the reverse triangular inequality in $\mathrm{L}^2(\prob_x)$
\[
\left|\sqrt{\sigma^2_n(f,x)} -\left|\frac 1 n \esp_x \left|\sum_{k=0}^{n-1} g(X_{k+1})- Pg(X_k) \right|^2\right|^{1/2} \right| \leqslant \frac 1 {\sqrt n} \left|\esp_x \left|g(X_0) - g(X_n) \right|^2 \right|^{1/2}
\]
and
\[
 \frac 1 {\sqrt n} \left|\esp_x \left|g(X_0) - g(X_n) \right|^2 \right|^{1/2} \leqslant \frac 1 {\sqrt n} |g(X_0)| + \frac 1 {\sqrt n} \left( \esp_x g^2(X_n)\right)^{1/2}
\]
But, lemma~\ref{lemma:u(X_n)/n} proves that if there is $\varepsilon \in \R_+^\star$ such that $\int_\R |y|^{4+\alpha+ \varepsilon} \di\rho(y)$ is finite and if we choose $\alpha'=1+ \alpha''$ with $\alpha''$ small enough, then
\[
\frac 1 n \esp_x g^2(X_n) \xrightarrow \, 0
\]
Moreover, $(\sum_{k=0}^{n-1} g(X_{k+1}) - Pg(X_k))_n$ is a martingales and so
\begin{flalign*}
\esp_x \left| \sum_{k=0}^{n-1} g(X_{k+1}) - Pg(X_k) \right|^2 &= \sum_{k=0}^{n-1} \esp_x (g(X_{k+1}) - Pg(X_k))^2 \\ &=\esp_x \sum_{k=0}^{n-1} P(g^2)(X_k) - (Pg)^2(X_k)
\end{flalign*}
and applying the first part of the proposition to $P(g^2)$ and to $(Pg)^2$, we get that
\[
\frac 1 n \sum_{k=0}^{n-1} Pg^2(X_{k}) - (Pg(X_k))^2 \xrightarrow\, \int g^2- (Pg)^2 \di\mu \; a.e.\text{ and in }\mathrm{L}^1(\prob_x)
\]
so, $\sigma_n^2(f,x)$ also converges to $\int g^2 - (Pg)^2 \di \mu$ and if $\sigma^2(f)\not=0$ we can apply corollary~\ref{corollary:TCL_martingales} to get the central limit theorem and the law of the iterated logarithm.
\end{proof}

\bibliographystyle{amsalpha}
\bibliography{biblio}
\end{document}